\documentclass{amsart}

\usepackage{amsfonts, amsmath, amssymb}

\newtheorem{theorem}{Theorem}[section]
\newtheorem{lemma}[theorem]{Lemma}
\newtheorem{cor}[theorem]{Corollary}
\newtheorem{fact}[theorem]{Fact}
\newtheorem{prop}[theorem]{Proposition}
\newtheorem{claim}[theorem]{Claim}

\theoremstyle{definition}
\newtheorem{example}[theorem]{Example}
\newtheorem{rmk}[theorem]{Remark}
\newtheorem{definition}[theorem]{Definition}
\newtheorem*{notation}{Notation}

\numberwithin{equation}{section}

\newcommand{\Pc}{\mathcal{P}}

\newcommand{\G}{\mathcal{G}}

\newcommand{\NN}{\mathbb{N}}

\newcommand{\MM}{\mathbb{M}}
\newcommand{\UU}{\mathbb{U}}

\DeclareMathOperator{\tp}{tp}
\DeclareMathOperator{\stp}{stp}
\DeclareMathOperator{\cb}{Cb}

\DeclareMathOperator{\dcl}{dcl}
\DeclareMathOperator{\acl}{acl}
\DeclareMathOperator{\Aut}{Aut}

\newcommand{\oa}{\overline{a}}
\newcommand{\ob}{\overline{b}}
\newcommand{\oc}{\overline{c}}
\newcommand{\od}{\overline{d}}
\newcommand{\oee}{\overline{e}}
\newcommand{\ox}{\overline{x}}

\DeclareMathOperator{\dom}{dom}
\DeclareMathOperator{\cod}{cod}
\DeclareMathOperator{\Mor}{Mor}
\DeclareMathOperator{\Ob}{Ob}
\DeclareMathOperator{\Hom}{Hom}
\DeclareMathOperator{\id}{id}

\def\Ind#1#2{#1\setbox0=\hbox{$#1x$}\kern\wd0\hbox to 0pt{\hss$#1\mid$\hss}
\lower.9\ht0\hbox to 0pt{\hss$#1\smile$\hss}\kern\wd0}

\def\ind{\mathop{\mathpalette\Ind{}}}

\def\notind#1#2{#1\setbox0=\hbox{$#1x$}\kern\wd0
\hbox to 0pt{\mathchardef\nn=12854\hss$#1\nn$\kern1.4\wd0\hss}
\hbox to 0pt{\hss$#1\mid$\hss}\lower.9\ht0 \hbox to 0pt{\hss$#1\smile$\hss}\kern\wd0}

\usepackage{tikz}
\usetikzlibrary{cd}

\title{Groupoids and Relative Internality}

\author{L\'{e}o Jimenez}
\address{L\'{e}o Jimenez\\
University of Notre Dame\\
Department of Mathematics\\
255 Hurley\\
Notre Dame, Indiana \  IN 46556\\
United States}
\email{ljimene4@nd.edu}

\date{\today}
\begin{document}

\begin{abstract}
In a stable theory, a stationary type $q \in S(A)$ internal to a family of partial types $\Pc$ over $A$ gives rise to a type-definable group, called its binding group. This group is isomorphic to the group $\Aut(q/\Pc,A)$ of permutations of the set of realizations of $q$, induced by automorphisms of the monster model, fixing $\Pc \cup A$ pointwise. In this paper, we investigate families of internal types varying uniformly, what we will call relative internality. We prove that the binding groups also vary uniformly, and are the isotropy groups of a natural type-definable groupoid (and even more). We then investigate how properties of this groupoid are related to properties of the type. In particular, we obtain internality criteria for certain 2-analysable types, and a sufficient condition for a type to preserve internality.
\end{abstract}

\maketitle

\setcounter{tocdepth}{1}
\tableofcontents

\section {Introduction}
\label{intro}

\noindent
In geometric stability theory, the notion of internality plays a central role, as a tool to understand the fine structure of definable sets. More recently, it has been developed outside of stable theories, using only stable embeddedness. In this paper, we will restrict ourselves to the stable context. More precisely, our basic setup will be the following: we work in a monster model $\MM$ of a stable theory $T$, eliminating imaginaries, and is given a family of partial types $\Pc$, all over some algebraically closed set of parameters $A$. A tuple $c$ is said to be a realization of $\Pc$ if it is a realization of some partial type in $\Pc$. We will often also write $\Pc$ for the set of realizations of $\Pc$ in $\MM$. A stationary type $q \in S(A)$ is said to be $\Pc$-internal if there are $B \supseteq A$, a realization $a$ of $q$, independent of $B$ over $A$, and a tuple $\oc$ of realizations of $\Pc$ (i.e. each realizing some type in $\Pc$) such that $a \in \dcl(\oc,B)$. It is said to be almost $\Pc$-internal if $a \in \acl(\oc,B)$ instead. The important part of this definition is the introduction of the new parameters $B$. The following result, which is Theorem 7.4.8 in \cite{pillay1996geometric}, produces a type-definable group action from this configuration: 

\begin{theorem}
\label{bindgrp}

Let $\MM$ be the monster model of a stable theory $T$, eliminating imaginaries. Suppose $q\in S(A)$ is internal to a family of types $\Pc$ over $A$, an algebraically closed set of parameters. Then there are an $A$-type-definable group $G$ and an $A$-definable group action of $G$ on the set of realizations of $q$, which is naturally isomorphic (as a group action), to the group $\Aut(q/\Pc,A)$ of permutations of the set of realizations of $q$, induced by automorphisms of $\MM$ fixing $\Pc \cup A$ pointwise. 

\end{theorem}


This is the result that we will generalize in this paper. The group arising in this theorem is called the binding group of $q$ over $\Pc$, and was first introduced by Zilber. At our level of generality, its existence was proved by Hrushovski. Our proof will follow very closely the proof given in \cite{pillay1996geometric}, with a few minor adjustments.

The binding group of $q$ over $\Pc$ will often be denoted by $\Aut(q/\Pc,A)$. It encodes the dependence on the extra parameters $B$. For example, if $B = A$, the binding group is trivial. For a more modern treatment of these binding groups, outside of stable theories, and with definable sets instead of types, we refer the reader to \cite{hrushovski2012groupoids} and \cite{haykazyan2017functoriality}. 

Recall that $\MM$ is a monster model of a stable theory $T$, eliminating imaginaries. If $\Phi$ is a partial type, we will denote $\Phi(\MM)$ the set of its realizations in $\MM$. 
Again, let $\Pc$ be a family of partial types over $A$ algebraically closed. Suppose there is a type $q \in S(A)$, and an $A$-definable function $\pi$, whose domain contains $q(\MM)$. 

\begin{definition}

The type $q$ is said to be relatively $\Pc$-internal via $\pi$ if for any $a \models q$, the type $\tp(a/\pi(a)A)$ is stationary and $\Pc$-internal. Denote this type $q _{\pi(a)}$.

\end{definition}

From this configuration, we can define a groupoid $\G$ (see section \ref{preliminaries} for a definition of groupoid). Its objects are realizations of $\pi(q)$, and for any $a,b \models q$, the set of morphisms $\Mor(\pi(a),\pi(b))$ consists of bijections from $q _{\pi(a)}(\MM)$ to $q  _{\pi(b)}(\MM)$, induced by automorphisms of $\MM$ fixing $\Pc \cup A$ pointwise, and taking $\pi(a)$ to $\pi(b)$. In particular, the isotropy groups $\Mor(\pi(a), \pi(a))$ are the binding groups $\Aut(q_{\pi(a)}/\Pc ,A)$, hence are type-definable  over $A\pi(a)$ by Theorem \ref{bindgrp}. This groupoid acts naturally on the set of realizations of $q$. By that we mean, setting $X = \lbrace (\sigma, a) \in \Mor(\G) \times q(\mathbb{M}): \dom(\sigma) = \pi(a) \rbrace$, that we have a map $X \rightarrow q(\mathbb{M})$, satisfying the obvious group action-like axioms. In our case, the action is given by $(\sigma,a) \rightarrow \sigma(a)$. We obtain the following generalization of the type-definability of binding groups: 

\begin{theorem}
\label{bindoid}

The groupoid $\G$ is isomorphic to an $A$-type-definable groupoid, and its natural groupoid action on realizations of $q$ is $A$-definable.

\end{theorem}

In particular, the binding groups are uniformly type-definable, and are the isotropy groups of the type-definable groupoid. This groupoid arises because the group $\Aut(q/\Pc)$, even when $q$ is not internal, still acts definably on the fibers of $\pi$, but its global action is not definable. The new and interesting fact is that all these local fiber actions are uniformly definable, and come together to form a type-definable groupoid.

In \cite{hrushovski2012groupoids} and \cite{haykazyan2017functoriality}, internality is considered in a different context, and from some internal sorts, a definable groupoid is constructed. It arises for different reasons, and we will compare the two groupoids in the next paragraphs. In these two papers, the authors fix a monster model $\mathbb{U}$ of some theory $T$, eliminating imaginaries, and a monster model $\mathbb{U}'$ of some theory $T'$, with $\mathbb{U} \subset \mathbb{U}'$. They assume that $\mathbb{U}$ is stably embedded in $\mathbb{U}'$, and $\UU '$ is internal to $\UU$ with only one new sort, called $S$. Under these assumptions, they construct a $*$-definable (over $\emptyset$) connected groupoid $\G'$, in $(\UU' )^{eq}$, with one distinguished object $a$ and a full $*$-definable (over $\emptyset$) subgroupoid $\G$ in $\UU$, such that $\Mor_{\G '}(a,a)$ acts definably on $\UU'$, and this action is isomorphic to $\Aut(\UU ' /\UU)$ acting on $S$. Note that in \cite{haykazyan2017functoriality}, the groupoid constructed is actually proven to be $\emptyset$-definable, under some mild additional assumption. 

The starting point of their proof is the following observation: since $\UU '$ is internal to $\UU$, there is a $b$-definable set $O_b$ in $\UU$, and a $c$-definable bijection $f_c: S \rightarrow O_b$. Roughly speaking, the idea is now to allow these parameters $b$ and $c$ to vary, the $b$ yielding objects of a groupoid, and the $f_c$ morphisms between objects. Therefore, this groupoid will encode the non-canonicity of the parameters $b$ and $c$ used to witness internality. The groupoid constructed in the present paper, however, encodes the fact that some maps are partially definable, but not globally definable.

Comparing these papers and ours, some questions arise. First, one can define relative internality in this different context, and it would be interesting to see if a groupoid witnessing it could be contructed there. Second, the groupoids obtained from internality live in the sort $\UU$, and this is used to obtain a correspondence between certain groupoids in $\UU$ and internal generalised imaginary sorts of $\UU$. In our setup, this would be equivalent to our groupoid $\G$ living in $\Pc ^{eq}$. As it will become clear from the proof, our groupoid does not live in $\Pc ^{eq}$. It would be desirable to identify some object in $\Pc ^{eq}$ coming from relative internality. In section \ref{construction}, we will discuss some obstruction to this. 

The rest of the paper will explore different properties of the groupoids arising from relative internality, and how they relate to the type $q$. Mostly, we will seek to link some properties of $\G$ to $\Pc$-internality, or almost $\Pc$-internality, of the type $q$.

One motivation for this is to be able to determine when an analysable type is in fact internal. Recall that a type $q$ is said to be $\Pc$-analysable in n steps if for any $a \models q$ there are $a_n = a, a_{n-1}, \cdots ,a_1$ such that $a_{i} \in \dcl(a_{i+1})$ and $\tp(a_{i+1}/a_i)$ is $\Pc$-internal for all $i$. Therefore, if $q$ is relatively $\Pc$-internal via $\pi$, and the type $\pi(q)$ is $\Pc$-internal, we see that $q$ is $\Pc$-analysable in two steps. The question of which analysable types are actually internal is connected with the Canonical Base Property, which is a property of finite U-rank theories. Introduced in \cite{pillay2002remarks}, it is a model theoretic translation of a result in complex geometry, and has some attractive consequences (see \cite{chatzidakis2012note}, \cite{pillay2003jet}). It states that for any tuple $a,b$, if $b = \cb(\stp(a/b))$, then $\tp(b/a)$ is almost internal to the family $\Pc$ of nonmodular U-rank one types. But it is proven in \cite{chatzidakis2012note} that this type is always $\Pc$-analysable. Therefore, the Canonical Base Property boils down to the collapse of an analysable type into an internal one. 

In this paper, we will expose two properties of groupoids implying that a relatively internal type is internal. The first one, retractability, was introduced in \cite{goodrick2010groupoids}, and is related to 3-uniqueness. Here, it will imply that the 2-analysable type is a product of two weakly orthogonal types, one of which is $\Pc$-internal. The second needs the construction of a Delta groupoid, which adds simplicial data to the groupoid. We define a notion of collapsing for Delta groupoids, which turns out to be equivalent, if $\pi(q)$ is $\Pc$-internal, to $\Pc$-internality of the type $q$. 

Finally, in \cite{moosa2010model}, a strengthening of internality, called being Moishezon, or preserving internality in later papers, was introduced, again motivated by properties of compact complex manifolds. A criterion for when an internal type preserves internality was proved in this paper, but under the assumption that the ambient theory has the Canonical Base Property. Here, we prove a criterion for preserving internality in terms of Delta groupoids, valid in any superstable theory.

The paper is organized as follow: in section 2, we recall some results concerning internality and stable theories that will be used frequently, and say a few words about groupoids. In section 3, we construct the type-definable groupoid of Theorem \ref{bindoid}. In section 4, we define retractability for a type-definable groupoid, and explore the consequences of this property. In Section 5, we define Delta groupoids, introduce the notion of collapsing, and link it with internality and preservation of internality.

Before we start, let us give a few conventions and notations. As stated before, we will work throughout in the monster model $\MM$ of a stable theory, eliminating imaginaries. Theorem \ref{bindoid} is stated over any small parameter set $A$, but we will work, without loss of generality, over the empty set. We also assume that $\acl(\emptyset) = \emptyset$. Recall that if $\Pc$ is a family of partial types over the empty set, by a realization of $\Pc$, we mean a tuple $c$ realizing some partial type in $\Pc$. We will often write $\Pc$ for the set of realizations of the family of partial type $\Pc$ in $\MM$, since no confusion could arise from this. Finally, recall that if $\Phi$ is any partial type, we will denote $\Phi(\MM)$ the set of its realizations in $\MM$. 

We assume familiarity with stability theory and geometric stability theory, for which \cite{pillay1996geometric} is a good reference.

I would like to thank my advisor, Anand Pillay, for giving me regular input and suggestions during the writing of this paper. I would also like to thank Levon Haykazyan, Rahim Moosa, and Omar L\'{e}on S\'{a}nchez for discussing the subject of this paper with me. Finally, I am grateful to my referee, whose comments and suggestions lead to a substantial improvement of this paper.

\bigskip
\section{Preliminaries}
\label{preliminaries}
\noindent

Internality of a type $q$ to a family of partial types $\Pc$ is equivalent to: there exists a set of parameters $B$ such that for any $a \models q$, there are $c_1, \cdots, c_n$ realizing $\Pc$ satisfying $a \in \dcl(c_1, \cdots ,c_n, B)$. Moreover, the parameters $B$ can be taken as realizations of $q$, as the following, which is Lemma 7.4.2 from \cite{pillay1996geometric}, shows:  

\begin{lemma}
\label{funsys}

Let $A$ be a small set of parameters. Suppose $\Pc$ is a family of partial types over $A$, and $q$ is a $\Pc$-internal stationary type over $A$. Then there exist a partial $A$-definable function $f(y_1, \cdots, y_m, z_1, \cdots,z_n)$, a sequence $a_1, \cdots, a_m$ of realizations of $q$, and a sequence $\Psi_1 ,\cdots ,\Psi_n$ of partial types in $\Pc$, such that for any $a$ realizing $q$, there are $c_i$ realizing $\Psi_i$, for $i = 1 \cdots n$, such that $a = f(a_1, \cdots, a_m, c_1, \cdots ,c_n)$.

\end{lemma}

The tuple $a_1, \cdots , a_m$ obtained in this lemma is called a fundamental system of solutions for $q$.

In fact, we define, for any type $q$:

\begin{definition}\label{funsol}

If $q$ is $\Pc$-internal, a tuple $\oa$ of realizations of $q$ is said to be a fundamental system of solutions for $q$ if for any $b \models q$, we have $b \in \dcl(\oa, \Pc)$. If $q$ has a fundamental system consisting of only one realization, it is said to be a fundamental type.

\end{definition}

The following fact will be used implicitly throughout the article:

\begin{fact}\label{fungroup}

If $q$ is internal to $\Pc$, and $r \in S(\emptyset)$ is the type of a fundamental system of solutions for $q$, then the binding groups $\Aut(q/\Pc)$ and $\Aut(r/\Pc)$ are $\emptyset$-definably isomorphic.

\end{fact}

By Lemma \ref{funsys}, any internal type has a fundamental system of solutions.  

\begin{rmk}
\label{funind}
By inspecting the proof (in \cite{pillay1996geometric}) of the previous lemma, one notices that the tuples $a_1, \cdots , a_n$ are independent realizations of $q$. This will be useful in section 5.
\end{rmk}

The following two facts will be useful to us:




\begin{fact}
\label{chuddcl}

For any family of partial types $\Pc$ and tuple $a$, we have $\tp(a/\dcl(a) \cap \Pc) \models \tp(a/\Pc)$.

\end{fact}

For a proof of this, see Claim II of the proof of Theorem 7.4.8 in \cite{pillay1996geometric}.

\begin{fact}
\label{tenzig}

If $\Pc$ is a family of partial types, for any two tuples $a$ and $b$, we have $\tp(a/\Pc) = \tp(b/\Pc)$ if and only if there is an automorphism of $\MM$, fixing $\Pc$, and taking $a$ to $b$.

\end{fact}

This can be proven adapting the proof of Lemma 10.1.5 in \cite{tent2012course}.

The main algebraic objects considered in this paper are groupoids. We recall their definition: 

\begin{definition}

A groupoid $\G$ is a non-empty category such that every morphism is invertible. 

\end{definition}

Therefore, a groupoid consists of two sets: a set of objects $\Ob(\G)$, and a set of morphisms $\Mor(\G)$. These are equipped with the partial composition on morphisms, and the domain and codomain maps. Moreover, for each object $a$, there is an identity map $\id_a \in \Mor(a,a)$.

Groupoids generalize groups. Indeed, every object of a groupoid $\G$ gives rise to the group $\Mor(a,a)$, called the isotropy group of $a$. But we also have the extra morphisms $\Mor(a,b)$, for any $a,b \in \Ob(\G)$. Remark that a group is then exactly a groupoid with only one object.

The set $\Mor(a,b)$ could be empty if $a \neq b$. This will actually have some meaningful model-theoretic content, and we can define: 

\begin{definition}

If $\G$ is a groupoid and $a \in \Ob(\G)$, then the connected component of $a$ is the set $\lbrace b \in \Ob(\G): \Mor(a,b) \neq \emptyset \rbrace$. A groupoid is connected if it has only one connected component, and totally disconnected if the connected component of any object is itself.

\end{definition}

Since we are interested in definable, or type-definable objects, we need to define these notions for groupoids.

\begin{definition}

A groupoid $\G$ is definable if the sets $\Ob(\G)$ and $\Mor(\G)$ are definable, and the composition, domain, codomain and inverse maps are definable. It is type-definable is these sets and maps are type-definable.

\end{definition}

\bigskip
\section{The construction of a groupoid}
\label{construction}
\noindent

Let $q \in S(\emptyset)$, a family of partial types $\Pc$ over $\emptyset$, and an $\emptyset$-definable function $\pi$, whose domain contains $q(\MM)$, such that $q$ is relatively $\Pc$-internal via $\pi$.

Remark that for any $a \models q$, the type $\tp(a/\pi(a))$ is implied by $q(x) \cup \lbrace \pi(x) = \pi(a) \rbrace$. We will denote this type $q_{\pi(a)}$. To ease notation, if $\oa$ is a tuple of realizations of $q$ with same image under $\pi$, we will denote $\pi(\oa)$ their common image. 

Recall that there is a groupoid $\G$, whose objects are given by $\pi(q)(\MM)$, and morphisms $\Mor(\pi(a),\pi(b))$ by the set of bijections from $q _{\pi(a)}(\MM)$ to $q _{\pi(b)}(\MM)$, induced by automorphisms of $\MM$ fixing $\Pc$ pointwise, and taking $\pi(a)$ to $\pi(b)$. Our goal is to prove this groupoid, as well as its action on realizations of $q$ (see in the introduction for a definition of a groupoid action) are $\emptyset$-type-definable. We now start the proof, which follows closely the proof of Theorem 7.4.8 from \cite{pillay1996geometric}:

\begin{proof}[Proof of Theorem~\ref{bindoid}]

First note that the objects are the $\emptyset$-type-definable set $\pi(q)$. So what we have to show is that the set of morphisms is $\emptyset$-type-definable, as well as domain and codomain maps, and composition. 

Note that since each $\pi$-fiber is $\Pc$-internal, we can apply Lemma \ref{funsys} to any of them, so each type $q_{\pi(a)}$ has a fundamental system of solution. The first step of the proof is to show that these fundamental systems can be chosen uniformly, in the following sense: 

\begin{claim}
\label{unif}

There exist a type $r$ over $\emptyset$, a partial $\emptyset$-definable function $f(y, z_1, \cdots,z_n)$, a sequence $\Psi_1 ,\cdots ,\Psi_n$ of partial types in $\Pc$. These satisfy that for each $\pi(a) \models \pi(q)$, there is $\oa \models r$ such that $\pi(a) = \pi(\oa)$, and for any other $a' \models q_{\pi(a)}$, there are $c_i$ realizing $\Psi_i$, for $i =  1 \cdots n$, with $a' = f(\oa, c_1, \cdots ,c_n)$.

\end{claim}

\begin{proof}

Let $\pi(a)$ be a realization of $\pi(q)$. Applying Lemma \ref{funsys} to $q_{\pi(a)}$ yields a partial $\pi(a)$-definable function $f(y_1, \cdots, y_m, z_1, \cdots,z_n)$, a sequence $a_1, \cdots, a_m$ of realizations of $q_{\pi(a)}$, and a sequence $\Psi_1 ,\cdots ,\Psi_n$ of partial types in $\Pc$, such that $q_{\pi(a)} \subset f(\oa, \Psi_1(\MM), \cdots , \Psi_n(\MM))$. 

Denote $\oa = (a_1, \cdots, a_m)$, and $r = \tp(\oa/\emptyset)$. Remark that since $\pi(\oa) = \pi(a) \in \dcl(\oa)$, the function $f$ is actually $\emptyset$-definable. By invariance, we see that $f,r$ and $\Psi_1, \cdots ,\Psi_n$ satisfy the required properties.

\end{proof}

We will now fix $r, f$ be as in Claim \ref{unif}, and $\Phi(\ox) = \Psi(x_1) \cup \cdots \cup \Psi(x_n)$. Fix $\pi(a)$, $\pi(b)$ and a realization $\oa$ of $r$ in $\pi^{-1}(\pi(a))$. Consider the set $X = \lbrace (\oa,\ob): \tp(\oa) = \tp(\ob) = r, \tp(\oa/\Pc) = \tp(\ob/\Pc) \rbrace$, it is the set we will use to encode morphisms. We have: 

\begin{claim}
\label{stabemb}
The set $X$ is $\emptyset$-type-definable.
\end{claim}

\begin{proof}

 Fact \ref{chuddcl} yields that $\tp(\oa/\dcl(\oa) \cap \Pc) \models \tp(\oa/\Pc)$. Consider the set $\lbrace \lambda_i(\ox): i \in I \rbrace$ of partial $\emptyset$-definable functions defined at $\oa$ with values in $\Pc$ (and these are the same at every realization of $r$). Then $\tp(\oa/\Pc) = \tp(\ob/\Pc)$ if and only if $\lambda_i(\oa) = \lambda_i(\ob)$ for all $i \in I$. Therefore $X = \lbrace (\oa,\ob): \tp(\oa) = \tp(\ob) = r, \lambda_i(\oa) = \lambda_i(\ob) \text{ for all } i \in I \rbrace$, which is an $\emptyset$-type-definable set.
\end{proof}

Let $r_{\oa} = \tp(\oa /\Pc)$. We then have the following: 

\begin{claim}
\label{biject}
The map from $\Mor(\pi(a), \pi(b))$ to $r_{\oa}(\MM) \cap \lbrace \ox: \pi(\ox) = \pi(b) \rbrace$ taking $\sigma$ to $\sigma(\oa)$ is a bijection.
\end{claim}

\begin{proof}

First injectivity: suppose $\sigma(\oa) = \tau(\oa)$. Every element of $\pi^{-1}(\pi(a)) \cap q(\MM)$ is written as $f(\oa,c)$, for some $c \models \Phi$, and $\sigma(f(\oa,c)) = f(\sigma(\oa),\sigma(c)) = f(\tau(a),c) = \tau(f(\oa,c))$, so $\tau = \sigma$.

For surjectivity, given $\ob \models r_{\oa}$, since $\oa$ and $\ob$ have the same type over $\Pc$, by Fact \ref{tenzig}, there is an automorphism of the monster model, fixing $\Pc$, and taking $\oa$ to $\ob$. The restriction of this automorphism to $q_{\pi(a)}(\MM)$ belongs to $\Mor(\pi(a), \pi(b))$.

\end{proof}

By Claim \ref{biject}, for any $(\oa,\ob) \in X$, there is a unique $\sigma \in \Mor(\pi(\oa), \pi(\ob))$ such that $\sigma(\oa) = \ob$. And for any $\oa \models r$ and $\sigma \in \Mor(\pi(\oa), \pi(\ob))$, we also have $(\oa,\sigma(\oa)) \in X$. However, this correspondence may not be injective: for a $\sigma \in \Mor(\pi(a),\pi(b))$, there are multiple elements of $X$ corresponding to it. We will solve this problem with an equivalence relation.

\begin{claim}
\label{formula}
There is a formula $\psi(\ox_1,\ox_2,y,z)$ such that for any $\sigma \in \Mor(\G)$, any $\oa \models r$, any $a \models q$ such that $\dom(\sigma) = \pi(a) = \pi(\oa)$ and any $b$, we have $\models \psi(\oa,\sigma(\oa),a,b)$ if and only if $b = \sigma(a)$. 


\end{claim}

\begin{proof}

By the proofs of Claim \ref{stabemb} and Claim \ref{unif}, if $\oa,\ob$ realise $r$, with $\lambda_i(\oa) = \lambda_i(\ob)$ for all $i$, and $c_1,c_2$ realise the partial type $\Phi$ of Claim \ref{unif}, then $f(\oa,c_1) = f(\oa,c_2)$ if and only if $f(\ob,c_1) = f(\ob,c_2)$ (and these are well defined). By compactness, there is a formula $\theta(w)$ and a finite subset $J \subset I$ such that the previous property is true replacing $\Phi$ by $\theta$ and $I$ by $J$. 

Let the formula $\psi(\ox_1,\ox_2,y,z)$ be $\exists w (f(\ox_1,w) = y \wedge f(\ox_2,w) = z \wedge \theta(w))$. We now check that this formula works. Suppose first that $\oa,\sigma(\oa),a,b$ satisfy it. Then there is $c \models \theta(w)$ such that $f(\oa,c) = a$ and $f(\sigma(\oa),c) = b$. 

But as $a \models q$, there is also $d \models \Phi$ such that $f(\oa, d) = a$. Since $d \models \Phi$, it is a realization of $\Pc$, hence $\sigma(a) = \sigma(f(\oa,d)) = f(\sigma(\oa),d)$. 

So $f(\oa, c) = a = f(\oa,d)$, the tuple $d$ is a realization of $\Phi$, and $f(\sigma(\oa,d)) = \sigma(a)$. By choice of $\psi$, this implies $b = f(\sigma(\oa),c) = f(\sigma(\oa),d) = \sigma(a)$.

Conversely, suppose that $b = \sigma(a)$. Since $\oa \models r$ and $\dom(\sigma) = \pi(a) = \pi(\oa)$, there is $c \models \Phi$ such that $f(\oa,c) = a$. Therefore $b = \sigma(f(\oa,c)) = f(\sigma(\oa), c)$, so we can take $c$ to be the $w$ of the formula.

\end{proof}

Now define an equivalence relation $E$ on $X$ as $(\oa_1,\ob_1)E(\oa_2,\ob_2)$ if and only if $\pi(\oa_1) = \pi(\oa_2)$ and for some $\sigma \in \Mor(\G)$, we have $\sigma(\oa_1) = \ob_1$ and $\sigma(\oa_2) = \ob_2$. So  $(\oa_1,\ob_1)E(\oa_2,\ob_2)$ if and only if the tuples $(\oa_1,\ob_1)$ and $(\oa_2,\ob_2)$ represent the same morphism $\sigma$. Then the following is true:

\begin{claim}
\label{eisdef}
$E$ is relatively $\emptyset$-definable on $X \times X$.
\end{claim}

\begin{proof}

Recall that we denote $\tp(a/\pi(a))$, for $a \models q$, by $q_{\pi(a)}$. We also denote $r_{\pi(a)} = \tp(\oa/\pi(\oa))$, for some $\oa \models r$ with $\pi(\oa) = \pi(a)$. 

We first show that $(\oa_1,\ob_1)E(\oa_2,\ob_2)$ if and only if $\pi(\oa_1) = \pi(\oa_2)$ and for any $a \models q_{\pi(\oa_1)}|_{\oa_1,\oa_2,\ob_1,\ob_2}$ we have $\forall z \psi(\oa_1,\ob_1,a,z) \leftrightarrow\psi(\oa_2,\ob_2,a,z)$. 

The left to right direction is immediate. So assume that the right-hand condition holds. There are $\sigma,\tau \in \Mor(\G)$ with $\sigma(\oa_1) = \ob_1$ and $\tau(\oa_2) = \ob_2$. Let $\oa_3 \models r_{\pi(\oa_1)}$ be independent from $\oa_1,\oa_2,\ob_1,\ob_2$. If we let $\oa_3 = (a_{3,1}, \cdots , a_{3,n})$, then by independence, we have $\forall z \psi(\oa_1,\ob_1,a_{3,i},z) \leftrightarrow\psi(\oa_2,\ob_2,a_{3,i},z)$, for all $1 \leq i \leq n$. Hence $\sigma(\oa_3) = \tau(\oa_3)$. Let $a'$ be any realization of $q_{\pi(\oa_1)}$. Then $a' = f(\oa_3,c)$ for some $c$, since $\oa_3$ is a realization of $r_{\pi(\oa_1)}$. So $\sigma(a') = \sigma(f(\oa_3,c)) = f(\sigma(\oa_3),c) = f(\tau(\oa_3),c) = \tau(a')$. This is true for any realization $a'$ of $q_{\pi(\oa_1)}$, so $\tau = \sigma$.

Notice that the right-hand condition is equivalent to a formula over $\pi(\oa_1)$ because the stationary type $q_{\pi(\oa_1)}$ is definable over $\pi(\oa_1)$. So if we fix $\pi(a) \models \pi(q)$, there is a formula $\theta_{\pi(a)}(z_1,t_1,z_2,t_2,y)$ over $\emptyset$ such that for any $\oa_1, \oa_2 \models r_{\pi(a)}$, and any $\ob_1,\ob_2$, we have $(\oa_1,\ob_1)E(\oa_2,\ob_2)$ if and only if $\theta_{\pi(a)}(\oa_1,\ob_1,\oa_2,\ob_2,\pi(a))$. A priori, this formula $\theta_{\pi(a)}$ depends on $\pi(a)$, hence we cannot yet conclude that $E$ is relatively definable, let alone relatively $\emptyset$-definable. However, if we can prove that for any $a,b \models q$ the formulas $\theta_{\pi(a)}(z_1,t_1,z_2,t_2,\pi(a))$ and $\theta_{\pi(b)}(z_1,t_1,z_2,t_2,\pi(a))$ are equivalent, we would get relative $\emptyset$-definability.

Note that the formula $\theta_{\pi(a)}$ we obtained is a defining scheme for a formula in the stationary type $q_{\pi(a)} = \tp(a/\pi(a))$. We will use this to show the desired equivalence.

Let $\pi(a),\pi(b) \models \pi(q)$, and $\sigma$ an automorphism such that $\sigma(\pi(a)) = \pi(b)$. Let $\phi(x,y,z)$ be a formula over $\emptyset$. Since $q_{\pi(a)}$ and $q_{\pi(b)}$ are definable and stationary, there are defining schemes $\theta_{\pi(a)}(z,\pi(a))$ (respectively $\theta_{\pi(b)}(z,\pi(b))$) for $\phi(x,y,z)$ and $q_{\pi(a)}$ (respectively $q_{\pi(b)}$), and the formulas $\theta_{\pi(-)}(z,y)$ are over the empty set. Now, let $\oc$ be a tuple, and $a'$ a realization of $q_{\pi(a)}\vert_{\oc}$, the unique non-forking extension of $q_{\pi(a)}$ to $\lbrace \pi(a),\oc \rbrace$. Then:

\begin{align*}
\theta_{\pi(a)}(\oc,\pi(a)) &\Leftrightarrow \phi(x,\pi(a),\oc) \in q_{\pi(a)}\vert_{\oc} \\
& \Leftrightarrow \models \phi(a',\pi(a),\oc) \\
& \Leftrightarrow \models \phi(\sigma(a'), \pi(b), \sigma(\oc)) \\
& \Leftrightarrow \phi(x,\pi(b),\sigma(\oc)) \in q_{\pi(b)}|_{\sigma(\oc)} \text{ because } \sigma(a') \models q_{\pi(b)}|_{\sigma(\oc)} \\
& \Leftrightarrow \theta_{\pi(b)}(\sigma(\oc),\pi(b)) \\
& \Leftrightarrow \theta_{\pi(b)}(\oc, \pi(a))
\end{align*}

Applying this to the formula $\forall w \psi(z_1,t_1,x,w) \leftrightarrow(\psi(z_2,t_2,x,w) \wedge \pi(z_1) = y = \pi(z_2))$ and the $q_{\pi(a)}$, where $z = (z_1,t_1,z_2,t_2)$, we obtain, for any $\pi(a) ,\pi(b)$ realizations of $\pi(q)$, that $\models \theta_{\pi(a)}(\oa_1,\ob_1,\oa_2,\ob_2,\pi(a))$ if and only if $\models \theta_{\pi(b)}(\oa_1,\ob_1,\oa_2,\ob_2,\pi(a))$. Therefore, we can fix $\pi(b)$, and use the formula $\theta _{\pi(b)}$ to obtain for any $\pi(a) \models \pi(q)$, for any $\oa_1, \oa_2 \models r_{\pi(a)}$ and any $\ob_1,\ob_2$, that $(\oa_1,\ob_1)E(\oa_2,\ob_2)$ if and only if $\theta_{\pi(b)}(\oa_1,\ob_1,\oa_2,\ob_2,\pi(a))$. So $\theta _{\pi(b)}$ is the formula defining $E$.

\end{proof}

Hence we obtain an $\emptyset$-type-definable set $X/E$. But we had, by Claim \ref{biject}, a map from $X$ to $\Mor(\G)$. And $(\oa_1,\ob_1)E(\oa_2,\ob_2)$ if and only if they have the same image under this map. Therefore we have obtained a bijection from $X/E$ to $\Mor(\G)$. Notice that this also yields $\emptyset$-definability of domain and codomain: since the maps are represented by elements in the fibers, we can just take images under $\pi$ of any of their representant.

We can, using this coding for morphisms of the groupoid, prove that the groupoip action is relatively $\emptyset$-definable. If $\sigma \in \Mor(\pi(a),\pi(b))$, we can pick any representant $(\oa, \sigma(\oa))$. Then $\sigma(a)$ is the unique tuple satisfying $\psi(\oa,\sigma(\oa),a,z)$. Since this does not depend on the representant we pick, we obtain that $\sigma(a) \in \dcl(\sigma,a)$ (and the formula witnessing it is uniform in $\sigma$ and $a$). This yields that the groupoid action is relatively $\emptyset$-definable.

To finish the proof, we need to construct the composition in an $\emptyset$-definable way. 

\begin{claim}
\label{compisdef}
The composition of $\Mor(\G)$ is definable.
\end{claim}

\begin{proof}

Let $\sigma,\tau,\mu \in \Mor(\mathcal{G})$. Let $\oa,\ob,\oc \models r$, with $\pi(\oa) = \dom(\sigma), \pi(\ob) = \dom(\tau)$ and $\pi(\oc) = \dom(\sigma)$. We will show that the equality $\tau \circ \sigma = \mu$ holds if and only if $\dom(\sigma) = \dom(\mu), \cod(\tau) = \cod(\mu),\cod(\sigma) = \dom (\tau) $ and for any $a \models q_{\pi(\oa)}|_{\oa,\ob,\oc,\sigma,\tau,\mu}$, we have:

$\forall z \psi(\oc,\mu(\oc),a,z) \leftrightarrow \exists u (\psi(\oa,\sigma(\oa),a,u) \wedge (\psi(\ob,\tau(\ob),u,z))) $

The left to right direction is again immediate. For the right to left direction, we can proceed as in Claim \ref{eisdef}, and assume that the right-hand side holds. Pick $\oa_2 \models r_{\pi(\oa)}|_{\sigma,\tau,\mu,\oa,\ob,\oc}$, then, as was done in Claim \ref{eisdef}, we obtain $\mu(\oa_2) = \tau \circ \sigma(\oa_2)$. But any $a' \models q_{\pi(\oa)}$ is equal to $f(\oa_2,c)$ for some $c$ tuple of realizations of $\mathcal{P}$. So we get $\mu(a') = f(\mu(\oa_2),c) = f(\tau \circ \sigma(\oa_2),c) = \tau \circ \sigma (a')$. So $\mu = \tau \circ \sigma$.

Note that since the type $q_{\pi(\oa)}$ is stationary and definable, the right-hand side condition is equivalent to a formula over $\dom(\sigma) = \pi(\oa)$. Moreover, the truth of this formula does not depend on the representants of $\sigma,\tau$ and $\mu$ that we pick. Therefore it only depends on $\sigma,\tau$ and $\mu$.

Hence, if we fix $\pi(a)$, we obtain a formula $\theta_{\pi(a)}(x,y,z)$ over $\pi(a)$ such that for all $\sigma,\tau,\mu \in \Mor(\G)$ with $\dom(\sigma) = \dom(\mu) = \pi(a)$, we have $\theta_{\pi(a)}(\sigma,\tau,\mu)$ if and only if $\mu = \tau \circ \sigma$. Again, this formula is a defining scheme for $q_{\pi(a)}$.

We can apply the proof of Claim \ref{eisdef} to this situation, to get a formula $\theta$ over $\emptyset$ such that for all $\sigma,\tau,\mu \in \Mor(\G)$, $\mu = \tau \circ \sigma$ if and only if $\theta(\sigma,\tau,\mu)$. So the composition in $\G$ is relatively $\emptyset$-definable.

\end{proof}
This finishes the proof: we have obtained a type-definable groupoid, and we already saw that its natural action on $q(\MM)$ is relatively $\emptyset$-definable.
\end{proof}

We will denote this groupoid $\G(q,\pi/\Pc)$, or just $\G$ when it is clear what type and projection are considered.

Here is a first connection between the groupoid $\G(q, \pi / \Pc)$ and the type $q$. In section \ref{preliminaries}, we defined the connected component of a groupoid. An easy consequence of Fact \ref{tenzig} is that the connected components of $\G(q,\pi/\Pc)$ correspond to the orbits of $\pi(q)$ under $\Aut(\pi(q)/\Pc)$ (even if this group is not type-definable).

What if $q$ is $\Pc$-internal? Our theorem specializes in the following way: we can pick $e \in \dcl(\emptyset)$, and set $\pi(a) = e$ for all $a \models q$. We obtain a groupoid with only one object, that is, a group, which is just the type-definable binding group of $q$ over $\Pc$. 

In the internal case, the type-definable group $\Aut(q/\Pc)$ can be shown (see \cite{pillay1996geometric}) to be definably isomorphic to a type-definable group in $\Pc^{eq}$, possibly using some extra parameters. In particular, the group $\Aut(q/\Pc)$ is internal to $\Pc$. One would hope that in our context, the groupoid $\G(q,\pi /\Pc)$ is also $\Pc$-internal. This result, proved with the help of Omar L\'{e}on S\'{a}nchez, shows that it is unfortunately not the case:

\begin{prop}
If $\G$ is internal to $\Pc$ and connected, then $q$ is internal to $\Pc$.
\end{prop}

\begin{proof}
By internality assumption, there is a set of parameters $B$ such that $\Mor(\G) \subset \dcl(\Pc,B)$. 

Let $a$ and $b$ be any realizations of $q$, and let $\oa$ be a fundamental system of solutions for $\tp(a/\pi(a))$. Since $\G$ is connected, there is $\sigma \in \Mor(\pi(a),\pi(b))$. Moreover, the tuple $\ob = \sigma(\oa)$ is a fundamental system of solutions for $\tp(b/\pi(b))$. Therefore, there is $\od \in \Pc$ such that $b = f(\ob,\od) = f(\sigma(\oa),\od)$. But $\sigma \in \dcl(\Pc,B)$, therefore $b \in \dcl(\Pc,B,\oa)$.
\end{proof}

Note that this result is also true if one replace connected by boundedly many connected components. Some connectedness assumption is required to make this proof work. 

In some cases, the groupoid associated to a non-internal type might actually be internal as well. Indeed, it is proven in \cite{haykazyan2017functoriality}, that the groupoids associated to certain relatively internal definable sets are internal to the base set. The context in which these objects are studied in this paper is slightly different from ours, and it would be interesting to see which results can transfer, in one way or the other.

\begin{rmk}
\label{recfun}

If we assume that $\pi(q)$ is $\Pc$-internal, Theorem \ref{bindgrp} yields the type-definable binding group $\Aut(\pi(q)/\Pc)$. If we moreover assume that $\pi(q)$ is fundamental (see Definition \ref{funsol}), we obtain a definable functor $\Pi: \G(q,\pi/\Pc) \rightarrow \Aut(\pi(q)/\Pc)$. Indeed, we can send the morphism represented by $(\oa,\ob)$ to the one represented by $(\pi(\oa), \pi(\ob))$. We will see in the fifth section that this generalizes to the case of $\pi(q)$ not fundamental, after introducing Delta groupoids.

\end{rmk}

\bigskip
\section{Retractability}
\label{retract}
\noindent

In this section, we consider retractability, which was introduced in \cite{goodrick2010groupoids}. There, it was used to study groupoids arising from internality, and was linked to 3-amalgamation in stable theories. Interestingly, it has some meaningful content in the context of our paper as well.

\begin{definition}

An $\emptyset$-type-definable groupoid $\mathcal{G}$ is retractable if it is connected and there exist an $\emptyset$-definable partial function $g(x,y)=g_{x,y}$ such that for all $a,b$ objects of $\G$, we have $g_{a,b} \in \Mor(a,b)$. Moreover, we require the compatibility condition that $g(b,c) \circ g(a,b) = g(a,c)$ for all objects $a,b,c$ (note that this implies $g_{a,a} = \id _a$ and $g_{a,b}^{-1} = g_{b,a}$ for all $a,b$). 
\end{definition}

The following was proved in \cite{goodrick2010groupoids}, but we include their proof here for completeness: 

\begin{rmk}
An equivalent definition of retractability is given by: there exist an $\emptyset$-type-definable group $G$, and a full, faithfull $\emptyset$-definable functor $F: \G \rightarrow G$.
\end{rmk}

\begin{proof}
If we have such a functor $F: \G \rightarrow G$, we can take $g_{a,b} = F^{-1}(\lbrace \id_{G} \rbrace) \cap \Mor(a,b)$, which is a singleton because $F$ is full and faithfull. The compatiblity condition is easily checked, and this is definable uniformly in $(a,b)$.

If $\G$ is retractable, then we can construct a relation $E$ on $\Mor(\G)$ as follows: if $\sigma \in \Mor(a,b)$ and $\tau \in \Mor(c,d)$, then $\sigma E \tau$ if and only if $\tau = g_{b,d} \circ \sigma \circ g_{c,a}$. By the compatibility condition, this is an equivalence relation, and it is $\emptyset$-definable. Now consider $G = \Mor(\G) /E$, and $F: \G \rightarrow G$ the quotient map. The groupoid law of $\G$ goes down to a group law on $G$. Indeed, if we want to compose $\sigma \in \Mor(a,b)$ and $\tau \in \Mor(c,d)$ in $G$, notice that $\tau E g_{d,a} \circ \tau$, so we can define $F(\sigma) \circ F(\tau) = F(\sigma \circ g_{d,a} \circ \tau)$. Again by the compatibility condition, this is well defined. Finally, it is easy to derive the group axioms from the groupoid axioms of $\G$.
\end{proof}

We are still working with a family of partial types $\Pc$ over the empty set, a type $q$, and an $\emptyset$-definable function $\pi$ such that $q$ is relatively $\Pc$-internal via $\pi$. Let $\G = \G(q,\pi/\Pc)$. Recall that we denote, for $a \models q$, the type $\tp(a/\pi(a))$ by $q_{\pi(a)}$.

In a stable theory, if $p$ and $q$ are stationary types over some fixed set of parameters $A$, the type of $(a,b)$, with $a \models p$ and $b \models q$ independent over $A$, is unique. We denote this type $p \otimes q$.

\begin{prop}
\label{retprod}
If $\G$ is retractable, then there is a complete type $p \in S(\emptyset)$, internal to $\Pc$, weakly orthogonal to $\pi(q)$, and an $\emptyset$-definable bijection between $q(\mathbb{M})$ and $p \otimes \pi(q) (\mathbb{M})$.
\end{prop}

\begin{proof}
We consider the $\emptyset$-definable relation $x E y \Leftrightarrow g_{\pi(x),\pi(y)}(x) = y$. The compatibility condition of retractability implies that this is an equivalence relation. Let $\rho$ be the quotient map, then $\rho(q)$ is a complete type over the empty set, and it will be the type $p$ of the proposition. 

There is an $\emptyset$-definable function $s: q(\mathbb{M}) \rightarrow \rho(q)(\mathbb{M}) \times \pi(q)(\mathbb{M})$ sending $x$ to $(\rho(x),\pi(x))$. Since $q$ is a complete type, $s(q(\MM))$ is the set of realizations of a complete type, denoted $s(q)$. But the function $s$ is bijective. Indeed, notice that each $E$-class has exactly one element in each fiber of $\pi$: each class has at least one element in a given fiber because $\G$ is connected, and no more than one because $g_{\pi(a), \pi(a)} = \id _a$. Therefore we can send $(\rho(a),\pi(b))$ to the unique element both in the $\pi(b)$ fiber and in the $E$-class of $a$, to obtain an inverse of $s$. So $\rho(q)(\mathbb{M}) \times \pi(q)(\mathbb{M})$ is in $\emptyset$-definable bijection with a complete type, hence is itself a complete type over the empty set. In particular $\pi(q)$ and $\rho(q)$ are weakly orthogonal, so $\rho(q)(\mathbb{M}) \times \pi(q)(\mathbb{M}) = \rho(q) \otimes \pi(q) (\mathbb{M})$. We denote $p = \rho(q)$.

We now just need to prove that $p$ is $\Pc$-internal. Each $E$-class has a unique representant in each $\pi$-fiber. Therefore, fixing $a \models q$, we have $p(\MM) \subset \dcl(q_{\pi(a)}(\MM))$. But by internality of the fibers, we get $q_{\pi(a)}(\MM) \subset \dcl(\oa, \Pc)$, for some tuple $\oa$. This yields $p(\MM) \subset \dcl(\oa,\Pc)$.

\end{proof}

\begin{cor}
\label{retint}

If $\G$ is retractable and $\pi(q)$ is $\Pc$-internal, then $q$ is $\Pc$-internal.

\end{cor}

Retractability yields a functor $F: \G \rightarrow G$, but one could ask if it has any consequence on the group $\Aut(q/\Pc)$. As it turns out, it does: 

\begin{prop}\label{retmorph}

If $\G$ is retractable, there is a morphism $R :  \Aut(q/\Pc) \rightarrow G$, which is surjective.

\end{prop}

\begin{proof}

We use the functor $F: \G \rightarrow G$. For $\sigma \in \Aut(q/\Pc)$, note that the restriction of $\sigma$ to $q_{\pi(a)}(\MM)$ is an element of $\Mor(\pi(a), \sigma(\pi(a)))$. We denote it by $\sigma \vert _{\pi(a)}$. We can then set $R(\sigma) = F(\sigma \vert _{\pi(a)})$. Let us show that $R$ is a surjective morphism $R : \Aut(q/\Pc) \rightarrow G$.

First, we need to prove that $R$ is well defined. To do so, we need to show that for any $b$, we have
$\sigma \vert _{\pi(b)} = g_{\sigma(\pi(a)), \sigma(\pi(b))} \circ \sigma \vert _{\pi(a)} \circ g_{\pi(b), \pi(a)}$, by definition of $F$.

Pick any $x$ with $\pi(x) = \sigma(\pi(a))$. Since $g_{\_,\_}$ is an uniformly $\emptyset$-definable family of partial functions, we have $g_{\sigma(\pi(a)), \sigma(\pi(b))}(x) = y$ if and only if $g_{\pi(a),\pi(b)} (\sigma ^{-1}(x)) = \sigma ^{-1}(y)$, for any $y$. Applying $\sigma$ to the second equality, we get, for all $y$, that $g_{\sigma(\pi(a)), \sigma(\pi(b))}(x) = y$ if and only if $\sigma(g_{\pi(a), \pi(b)}(\sigma ^{-1}(x))) = y$, which yields that $\sigma \vert _{\pi(b)} \circ g_{\pi(a), \pi(b)} \circ \sigma \vert _{\pi(a)}^{-1} = g_{\sigma(\pi(a)), \sigma(\pi(b))}$, what we wanted.

Therefore we have a well defined map $R: \Aut(q/\Pc) \rightarrow G$. It is a morphism because: 

\begin{align*}
    R(\sigma \circ \tau) & = F((\sigma \circ \tau) \vert _{\pi(a)}) \\
    & = F(\sigma \vert _{\tau(\pi(a))} \circ \tau \vert _{\pi(a)}) \\
    & = F(\sigma \vert _{\tau(\pi(a))}) \circ F(\tau \vert _{\pi(a)}))\\
    & = R(\sigma) \circ M(\tau)
\end{align*}

For surjectivity, by fullness of $F$, it is enough to prove that for $\sigma \in \Mor(\pi(a), \pi(a))$, there is $\tau \in \Aut(q/\Pc)$ restricting to $\sigma$. This is true by definition of $\Mor(\pi(a), \pi(a))$.

\end{proof}

\begin{prop}\label{retiso}

The group $G$ witnessing retractability is relatively $\emptyset$-definably isomorphic to $\Aut(p/\Pc)$, the binding group of $p$ over $\Pc$ (where $p$ is the type of Proposition \ref{retprod}).

\end{prop}

\begin{proof}

Recall that $\Mor(\G)$ is given by $X/E$, where $X$ is an $\emptyset$-type-definable set, and $E$ is an $\emptyset$-definable equivalence relation. Moreover, the type-definable set $X$ is composed of pairs of realizations of $r$, the type introduced in the proof of Theorem \ref{bindoid}. In the proof Proposition \ref{retprod}, we constructed an $\emptyset$-definable quotient map $\rho: q(\MM) \rightarrow p(\MM)$. The type $p = \rho(q)$ is $\Pc$-internal, hence its binding group $\Aut(p/\Pc)$ is similarly given by the type $r'$ of a fundamental system of solutions, an $\emptyset$-type-definable set $X'$ and an $\emptyset$-definable equivalence relation $E'$. We can assume that $r' = \rho(r)$.

For any $\oa \models r$, this allows us to define a group morphism:

\begin{align*}
P_{\pi(\oa)} : \Aut(\tp(\oa/\pi(\oa))/\Pc) & \rightarrow \Aut(p/\Pc) \\
\sigma = (\oa, \sigma(\oa))/E & \rightarrow (\rho (\oa), \rho(\sigma (\oa)))/E'
\end{align*}

\noindent and by construction of $\rho$, this is an isomorphism. It is relatively $\oa$-definable.

We are also given, by the retractability assumption, a relatively $\emptyset$-definable full and faithfull functor $F : \Mor(\G) \rightarrow G$. By restriction this yields, for any $\oa \models r$, a relatively $\pi(\oa)$-definable group isomorphism $F_{\pi(\oa)} : \Aut(\tp(\oa/\pi(\oa))/\Pc) \rightarrow G$. 

Hence, for any $\oa \models r$, the groups $G$ and $\Aut(p/\Pc)$ are relatively $\oa$-definably isomorphic via the composition $P_{\pi(\oa)} \circ F_{\pi(\oa)}^{-1}$. To complete the proof, we need to show that this morphism is actually relatively $\emptyset$-definable. To do so, it is enough (via a compactness argument) to prove that the graph of $P_{\pi(\oa)} \circ F_{\pi(\oa)}^{-1}$ is fixed by any automorphism of $\MM$. 

\begin{claim}\label{unifmor1}

For any $\oa,\ob \models r$ and $g \in G$, we have $P_{\pi(\ob)} \circ F_{\pi(\ob)}^{-1}(g) = P_{\pi(\oa)} \circ F_{\pi(\oa)}^{-1} (g)$.

\end{claim}

\begin{proof}

By the proof of Proposition \ref{retmorph}, if $\oa,\ob$ are realizations of $r$ and $g \in G$, then there is $\sigma \in \Aut(q/\Pc)$ such that $F_{\pi(\oa)}^{-1}(g)$ is the restriction of $\sigma$ to $\Aut(\tp(\oa/\pi(\oa))/\Pc)$ and $F_{\pi(\ob)}^{-1}(g)$ is the restriction of $\sigma$ to $\Aut(\tp(\ob/\pi(\ob))/\Pc)$. Hence $F_{\pi(\oa)}^{-1}(g) = (\oa,\sigma(\oa))/E$ and $F_{\pi(\ob)}^{-1}(g) = (g_{\pi(\oa), \pi(\ob)}(\oa), \sigma(g_{\pi(\oa), \pi(\ob)}(\oa)))/E$, as $g_{\pi(\oa), \pi(\ob)}(\oa) \models r$. We then obtain:

\begin{align*}
    P_{\pi(\ob)} \circ F_{\pi(\ob)}^{-1}(g) & = P_{\pi(\ob)}((g_{\pi(\oa), \pi(\ob)}(\oa), \sigma(g_{\pi(\oa), \pi(\ob)}(\oa)))/E) \\
    & = (\rho(g_{\pi(\oa), \pi(\ob)}(\oa)), \rho(\sigma(g_{\pi(\oa), \pi(\ob)}(\oa))))/E' \\
    & = (\rho(g_{\pi(\oa), \pi(\ob)}(\oa)), \sigma(\rho(g_{\pi(\oa), \pi(\ob)}(\oa))))/E' \\
    & = (\rho(\oa), \sigma(\rho(\oa)))/E'\text{ by definition of } \rho \\
    & = P_{\pi(\oa)}(\oa,\sigma(\oa)/E) \\
    & = P_{\pi(\oa)} \circ F_{\pi(\oa)}^{-1} (g)
\end{align*}

\end{proof}

Now let $g \in G$, let $(g, P_{\pi(\oa)} \circ F_{\pi(\oa)}^{-1}(g))$ be in the graph of $P_{\pi(\oa)} \circ F_{\pi(\oa)}^{-1}$, and let $\mu$ be an automorphism of $\MM$. We want to show that $\mu(g, P_{\pi(\oa)} \circ F_{\pi(\oa)}^{-1}(g))$ is also in the graph of $P_{\pi(\oa)} \circ F_{\pi(\oa)}^{-1}$. 

\begin{claim}\label{unifmor2}

We have $\mu(F_{\pi(\oa)}) = F_{\pi(\mu(\oa))}$. 

\end{claim}

\begin{proof}

This is because the maps $F_{\pi(\oa)}^{-1}$ are uniformly $\pi(\oa)$-definable.

\end{proof}

\begin{claim}\label{unifmor3}

For any $\sigma \in \Aut(\tp(\oa/\pi(\oa))/\Pc)$, we have $\mu(P_{\pi(\oa)}(\sigma)) = P_{\pi(\mu(\oa))}(\mu(\sigma))$.

\end{claim}

\begin{proof}

The set $\Mor(\G)$ is $\emptyset$-type-definable, hence for any $\sigma \in \Aut(\tp(\oa/\pi(\oa))/\Pc)$ we have $\mu(\sigma) = \tau \in \Mor(\G)$. In particular, we obtain $\mu(\sigma(\oa)) = \tau(\mu(\oa))$, which yields: 

\begin{align*}
    \mu(P_{\pi(\oa)}(\sigma)) & = \mu(P_{\pi(\oa)}((\oa,\sigma(\oa))/E)) \\
    & = (\rho(\mu(\oa)),\rho(\mu(\sigma(\oa))))/E' \\
    & = (\rho(\mu(\oa)),\rho(\tau(\mu(\oa))))/E' \\
    & = P_{\pi(\mu(\oa))}((\mu(\oa)), \tau(\mu(\oa))/E) \\
    & = P_{\pi(\mu(\oa))}(\tau) \\
    & = P_{\pi(\mu(\oa))}(\mu(\sigma)) 
\end{align*}

\end{proof}

Putting everything together, we obtain: 

\begin{align*}
    \mu(P_{\pi(\oa)} \circ F_{\pi(\oa)}^{-1}(g)) & = \mu(P_{\pi(\oa)}) \circ \mu(F_{\pi(\oa)})^{-1}(\mu(g)) \\
    & = P_{\pi(\mu(\oa))} \circ F_{\pi(\mu(\oa))}(\mu(g)) \text{ by Claims } \ref{unifmor2} \text{ and } \ref{unifmor3}\\ 
    & = P_{\pi(\oa)} \circ F_{\pi(\oa)}^{-1}(\mu(g)) \text{ by Claim }\ref{unifmor1}
\end{align*}

\noindent so $(\mu(g),P_{\pi(\oa)} \circ F_{\pi(\oa)}^{-1}(\mu(g)))$ belongs to the graph of $P_{\pi(\oa)} \circ F_{\pi(\oa)}^{-1}$, what we needed to prove.

\end{proof}

If $\pi(q)$ is $\Pc$-internal, it has an $\emptyset$-type-definable binding group $\Aut(\pi(q)/\Pc)$, and we have: 

\begin{theorem}\label{retgrp}
If $\G$ is retractable and $\pi(q)$ is $\Pc$-internal and fundamental, then $q$ is $\Pc$-internal and $\Aut(q/\Pc)$ is $\emptyset$-definably isomorphic to $G \times \Aut(\pi(q)/\Pc)$.
\end{theorem}

\begin{proof}

We know from Corollary \ref{retint} that $q$ is internal. Let $\oa$ be a fundamental system of solutions for $q$. 

Recall that there are two $\emptyset$-definable quotient maps $\pi : q(\MM) \rightarrow \pi(q)(\MM)$ and $\rho : q(\MM) \rightarrow p(\MM) = \rho(q)(\MM) $. The tuples $\pi(\oa)$ and $\rho(\oa)$ are fundamental systems of solutions for $\pi(q)$ and $\rho(q)$. As was done in Proposition \ref{retiso}, we can use this to construct two $\oa$-definable surjective group morphisms $\overline{\pi} : \Aut(q/\Pc) \rightarrow \Aut(\pi(q)/\Pc)$ and $\overline{\rho} : \Aut(q/\Pc) \rightarrow \Aut(\rho(q)/\Pc)$. Using techniques similar to the ones in Proposition \ref{retiso}, we can prove that these two morphisms are $\emptyset$-definable. 

Hence we have produced two $\emptyset$-definable group morphisms $\overline{\pi} : \Aut(q/\Pc) \rightarrow \Aut(\pi(q)/\Pc)$ and $\overline{\rho} : \Aut(q/\Pc) \rightarrow \Aut(\rho(q)/\Pc)$, both surjective. To obtain the desired isomorphism, it would be enough to prove that $\ker (\overline{\pi}) \cap \ker (\overline{\rho}) = \id$ and that any element of $\Aut (q/\Pc)$ can be written as the product of an element of $\ker (\overline{\rho})$ and an element of $\ker (\overline{\pi})$. 

Suppose that $\sigma \in \ker(\overline{\pi}) \cap \ker (\overline{\rho})$, and let $a \models q$. Then $\sigma$ fixes $\pi^{-1} \{ \pi(a) \}$ setwise. But $\sigma \in \ker(\overline{\rho})$, hence must fix $\pi^{-1} \{ \pi(a) \}$ pointwise. Since this is true for any $a \models q$, we conclude that $\sigma = \id$.

Let $\sigma $ be any morphism in $\Aut(q/\Pc)$ and $a \models q$. Consider $g_{\pi(a), \pi(\sigma(a))} \in \G$. It extends to an automorphism $\tau \in \Aut(q/\Pc)$ by Fact \ref{tenzig}, which has to belong to $\ker(\overline{\rho})$. We can write $\sigma = \tau \circ \tau^{-1} \circ \sigma$, so we only need to prove that $\tau^{-1} \circ \sigma \in \ker(\overline{\pi})$. But $\pi(\tau^{-1} \circ \sigma(a)) = \pi(a)$ and $\pi(q)$ is fundamental, so this implies $\overline{\pi}(\tau^{-1} \circ \sigma ) = \id$.

\end{proof}

\begin{rmk}

The assumption that $\pi(q)$ is fundamental seems necessary for this proof to go through. We still do not know if this theorem is valid without that assumption.

\end{rmk}

So retractability of the groupoid gives a lot more than just internality of the type. In fact, internality does not imply retractability, even if the groupoid is connected. 

\begin{example}

Consider the two sorted structure $\MM = (G,X, \mathcal{L}_G, *)$ with one sort being a connected stable group $G$ in the language $\mathcal{L}_G$, and the other sort being a principal homogeneous space $X$ for $G$, with group action $*$. We will work in $\MM^{eq}$. 

One can quickly prove that the sort $X$ has only one 1-type $q$ over $\emptyset$, and that this type is stationary and internal to $G$, with binding group isomorphic to $G$. 

Assume that there is an $\emptyset$-definable normal subgroup $H$ of $G$, such that the short exact sequence: \\

$1 \rightarrow H \rightarrow G \rightarrow G/H \rightarrow 1$ \\

\noindent does not definably split.

The group action of $G$ on $X$ defines an equivalence relation $E$, where the class of an element $a \in X$ is its orbit $H * a$. Hence, we can define a map $\pi : X \rightarrow X/E$, sending $a \in X$ to $H * a$. This is $\emptyset$-definable, we have $\Aut(\pi(q)/G) \cong G/H$ and for any $a$, that $\Aut(\tp(a/\pi(a)) = H$. The type $q$ is relatively $G$-internal via $\pi$, yielding a groupoid $\G$. Since $G/H$ acts transitively on $\pi(q)$, this groupoid is connected. Moreover, we have a definable short exact sequence: \\

$1 \rightarrow H \rightarrow G \rightarrow G/H \rightarrow 1$ \\

\noindent which, by assumption, is not definably split. However, if $\G$ was retractable, our previous work implies that this sequence would be definably split. Hence $\G$ is not retractable, even though $q$ is $G$-internal.

\end{example}

In the next section, we will introduce a necessary and sufficient condition for internality, using Delta groupoids.

\bigskip
\section{Delta groupoids and collapsing}
\label{coversection}
\noindent

In this section, we are again working with a family of partial types $\Pc$ over the empty set, a type $q \in S(\emptyset)$, and an $\emptyset$-definable function $\pi$ such that $q$ is relatively $\Pc$-internal via $\pi$. We have obtained a groupoid from this relatively internal type. But since $q$ is stationary, for any $n$, we can form the product of $q$ with itself $n$-times, denoted $q^{(n)}$. We still have an $\emptyset$-definable projection map $\pi^{(n)}$, given by applying $\pi$ on each coordinate, and its fibers are $\Pc$-internal too. All we are missing to get relative internality and apply Theorem \ref{bindoid} is that the type of a fiber be stationary. This is an easy application of forking calculus. 

\begin{fact}
For any $n$ and $(a_1, \cdots ,a_n) \models q^{(n)}$, $\tp(a_1, \cdots, a_n/\pi(a_1), \cdots , \pi(a_n))$ is stationary.

\end{fact}

Hence for each $n \geq 1$, the type $q^{(n)}$, together with the map $\pi^{(n)}$, satisfies the assumptions of Theorem \ref{bindoid}. We therefore obtain a sequence $\mathcal{G}_n$ of $\emptyset$-type-definable groupoids. Our first groupoid $\mathcal{G}$, associated to $q$ and $\pi$, becomes $\mathcal{G}_1$ in this new notation.

Recall that $\mathcal{G}$ was constructed using a type $r$, which corresponds to a fundamental system of solutions of the type $q_{\pi(a)} = \tp(a/\pi(a))$, for some (any) $a \models q$. Morphisms of $\G$ were then obtained as elements of $X/E$, where $X$ is an $\emptyset$-type-definable subset of $r(\MM)^2$, and $E$ is a relatively $\emptyset$-definable equivalence relation on $X \times X$.

Notice that for any $\oa \models q^{(n)}$, the type $\tp(\oa/\pi(\oa))$ has a fundamental solution that is a realization of $r^{(n)}$. Hence, for each $n$, the type $r^{(n)}$ will play for $q^{(n)}$ and $\pi^{(n)}$ the same role as $r$ for $q$ and $\pi$. Thus we obtain, for each $n$, an $\emptyset$-type-definable subset $X_n \subset r^{(n)}(\MM)^2$ and a relatively $\emptyset$-definable equivalence relation $E_n$ on $X_n \times X_n$, such that $\Mor(\G_n)$ is given by $X_n / E_n$.

This yields $\emptyset$-definable functors between the $\G_n$. To see this, let us introduce some notation: if $\oa = (a_1, \cdots , a_n)$ is a tuple, then for any $1 \leq i \leq n$, we denote $\oa^{\wedge i} = (a_1,\cdots, \hat{a_i}, \cdots a_n)$ where the hat means the corresponding coordinate has been removed. Now, if $n>1$, an element $\sigma$ of $\Mor(\G_n)$ corresponds to the $E_n$-class of $(\oa, \ob) = ((a_1, \cdots , a_n), (b_1, \cdots , b_n))$, where $\oa$ and $\ob$ are realizations of $r^{(n)}$. For any $1 \leq i \leq n $, we can then send $(\oa, \ob)/E_n$ to $(\oa^{\wedge i},\ob^{\wedge i})/E_{n-1}$. This is well defined, as $(\oa^{\wedge i} ,\ob^{\wedge i}) \in X_{n-1}^2$, and $\emptyset$-definable.

For each $n>1$ and each $1 \leq i \leq n$, we hence obtain $\emptyset$-definable maps: 

\begin{align*}
    \partial_i^n : \Mor(\G_n) & \rightarrow \Mor(\G_{n-1}) \\
    (\oa ,\ob)/E_n & \rightarrow (\oa ^{\wedge i} ,\ob ^{\wedge i})/E_{n-1}
\end{align*}

\noindent and by setting $\partial_i^n(\pi(\oa)) = \pi(\oa ^{\wedge i})$, we can easily check that each $\partial_i^n$ is an $\emptyset$-definable functor from $\G_{n}$ to $\G_{n-1}$.

These functors have a clear interpretation as restrictions of partial automorphisms. Indeed, if $\oa = (a_1, \cdots ,a_n) \models r^{(n)}$ and $\ob = (b_1, \cdots ,b_n) \models r^{(n)}$, then an element $\sigma$ of $\Hom(\G_n)(\pi(\oa),\pi(\ob))$ is a bijection: 

\begin{align*}
    \sigma : \tp(\oa/\pi(\oa))(\MM) \rightarrow \tp(\ob/\pi(\ob))(\MM)
\end{align*}

\noindent which is the restriction of an automorphism of $\MM$ fixing $\Pc$ pointwise. The element $\partial_i^n(\sigma)$ of $\G_{n-1}$ is then the restriction of $\sigma$ to a bijection:
 
 \begin{align*}
    \partial_i^n(\sigma) : \tp(\oa^{\wedge i}/\pi(\oa^{\wedge i}))(\MM) \rightarrow\tp(\ob^ {\wedge i }/\pi(\ob^ {\wedge i }))(\MM)
 \end{align*}
 
 \noindent which still is the restriction of the same global automorphism.

\begin{rmk}
If we assume that $\pi(q)$ is $\Pc$-internal, then by Remark \ref{funind}, there is some $n$ such that any $n$ independent realizations of $\pi(q)$ form a fundamental system of solutions. Therefore we obtain, as was done in Remark \ref{recfun}, a definable functor $F: \G_n \rightarrow \Aut(\pi(q)^{(n)}/\Pc)$.
\end{rmk}

We are now ready to define the algebraic structure of interest, which will be an $\emptyset$-type-definable Delta groupoid.

\begin{definition}

A Delta groupoid is the following data: 

\begin{enumerate}

\item For every integer $n \in \NN \backslash \lbrace 0 \rbrace$, a groupoid $\G _n$
\item For every integer $n \in \NN \backslash \lbrace 0,1 \rbrace$, and every $i \in \lbrace 0, \cdots ,n \rbrace$, a groupoid morphism (that is, a functor) $\partial_i^n: \G_{n} \rightarrow \G_{n-1}$, called a face map

\end{enumerate}

\noindent subject to the following condition:

\item $\partial_i^{n} \circ \partial_j^{n+1} = \partial_{j-1}^n \circ \partial_i^{n+1}$ for all $i<j \leq n$ and $n \geq 1$.

\end{definition}

Note that this definition, while adapted to our purpose, is not the one usually given in the simplicial homotopy literature. The interested reader can find an alternative category-theoretic definition in \cite{friedman2012survey}.

\begin{definition}

A Delta groupoid $\G$ is $\emptyset$-type-definable if every groupoid $\G_n$ is $\emptyset$-type-definable, and all the face maps are $\emptyset$-type-definable. 

\end{definition}

The previously defined groupoids $\G _n$ and maps $\partial$ are then easily checked to form an $\emptyset$-type-definable Delta groupoid. We will denote it by $\G$ (the previously constructed groupoid now becomes $\G_1$). Remark that the $\G _n$ are not uniformly type-definable (they do not even live in the same sorts). 

\begin{notation}
If $\oa \models q^{n}$ for some $n$, then the type $\tp(\oa/\pi(\oa))$ is $\Pc$-internal, and we will denote $G_{\pi(\oa)}$ its binding group. It is $\Mor(\pi(\oa),\pi(\oa))$ in $\G_n$.
\end{notation}

Using the Delta groupoid structure, the data of the $G_{\pi(\oa)}$ can be formed into a projective system of type-definable groups. Indeed, we can take our directed set to be $\lbrace \pi(\oa): \pi(\oa) \models \pi^n(q^{(n)}) \text{ for some } n \rbrace$, with $(\pi(a_1), \cdots ,\pi(a_n)) \leq (\pi(b_1), \cdots ,\pi(b_m))$ if and only if $n \leq m$ and $\pi(a_i) = \pi(b_i)$ for all $i \leq n$. If $\pi(\oa) \leq \pi(\ob)$, the restriction map $G_{\pi(\ob)} \rightarrow G_{\pi(\oa)}$ is definable, as it is a composition of face maps. These maps, together with the $G_{\pi(\oa)}$, are easily checked to form a projective system. In particular, we obtain the projective limit $\varprojlim G_{\pi(\oa)}$.

\begin{definition}
The Delta groupoid $\G$ is said to collapse if there is a tuple $\oa$ of independent realizations of $q$ such that for any $\ob \geq \oa$, the map $G_{\pi(\ob)} \rightarrow G_{\pi(\oa)}$ is injective. It is said to almost collapse if the maps $G_{\pi(\ob)} \rightarrow G_{\pi(\oa)}$ have finite kernel instead.
\end{definition}

These maps $G_{\pi(\ob)} \rightarrow G_{\pi(\oa)}$ are not necessarily surjective, but some will be if $\pi(q)$ is $\Pc$-internal:

\begin{rmk}\label{facesurj}

If $\pi(q)$ is $\Pc$-internal, then there is $m \in \mathbb{N}$ such that for all $n \geq m$, all $\pi(\oa) \models q^{(n)}$ and $\pi(\oa) \leq \pi(\ob)$, the map $G_{\pi(\ob)} \rightarrow G_{\pi(\oa)}$ is surjective. 

\end{rmk}

\begin{proof}

Let $\oa_0 \models q^{(m)}$ be such that $\pi(\oa_0)$ is a fundamental system of solutions for $\pi(q)$ (such an $\oa_0$ exist by Remark \ref{funind}). Then any $m$ independent realizations of $\pi(q)$ will be a fundamental system a solutions. Hence for any $n \geq m$ and any $\oa \models q^{(n)}$, the tuple $\pi(\oa)$ is a fundamental system of solutions for $\pi(q)$.

Fix $\oa \models q^{(n)}$ for $n \geq m$ and $\pi(\ob) \geq \pi(\oa)$, consider the map $G_{\pi(\ob)} \rightarrow G_{\pi(\oa)}$. Let $\sigma \in G_{\pi(\oa)}$, it is the restriction to $\tp(\oa/\pi(\oa))(\MM)$ of an automorphism $\Tilde{\sigma}$ of $\MM$. But $\pi(\oa)$ is a fundamental system of solutions for $\pi(q)$, and $\Tilde{\sigma}$ fixes $\pi(\oa)$. Hence $\Tilde{\sigma}$ fixes $\pi(q)(\MM)$, and in particular fixes $\pi(\ob)$. Therefore $\Tilde{\sigma}$ restricts to an element of $G_{\pi(\ob)}$, and the image of this element under $G_{\pi(\ob)} \rightarrow G_{\pi(\oa)}$ has to be $\sigma$.

\end{proof}

We will now prove a very useful equivalent condition. 

\begin{lemma}\label{colleqind}

The Delta groupoid associated to $q, \pi$ and $\Pc$ collapses (respectively almost collapses) if and only if there is a tuple $\oa$ of independent realizations of $r$ such that for any (some) $b \models q$, independent of $\oa$, we have $b \in \dcl(\oa, \pi(b), \Pc)$ (respectively $b \in \acl(\oa, \pi(b), \Pc)$). 

\end{lemma}

\begin{proof}

We will only prove the equivalence for collapsing, the other equivalence being proved in a similar way. Suppose first that there is a tuple $\oa$ of realizations of $r$ such that for any $b \models q$ independent of $\oa$, we have $b \in \dcl(\oa, \pi(b), \Pc)$. Let $\pi(\ob) > \pi(\oa)$, these are tuples of independent realizations of $\pi(q)$, we want to prove that $G_{\pi(\ob)} \rightarrow G_{\pi(\oa)}$ is injective. The type $\tp(\ob /\pi(\ob))$ is $\Pc$-internal, hence it has a fundamental system of solutions $(b_1, \cdots ,b_n)$. Each of these $b_i$ is either in $\pi^{-1}(\pi(\oa))$, and hence in $\dcl(\oa, \Pc)$, or $\pi(b_i)$ is independent of $\pi(\oa)$ over $\emptyset$, and we can then assume $b_i$ to be independent of $\oa$ over $\emptyset$. In this second case, the assumption yields $b_i \in \dcl(\oa, \pi(b_i), \Pc)$. Hence we obtain $b_i \in \dcl(\oa, \pi(b_i), \Pc)$ for all $i$, so $\tp(\ob/\pi(\ob))(\MM) \subset \dcl(\oa, \pi(\ob), \Pc)$. 

Now let $\sigma \in G_{\pi(\ob))}$ be such that its image under $G_{\pi(\ob)} \rightarrow G_{\pi(\oa)}$ is the identity. Then it has to fix $\oa$, and it fixes $\pi(\ob)$ and $\Pc$ too. Hence it has to fix $\tp(\ob/\pi(\ob))(\MM)$, so it is the identity of $G_{\pi(\ob)}$.

For the other implication, suppose that the Delta groupoid collapses. Hence there is a tuple $\oa$ of independent realizations of $q$ such that for any $\pi(\ob) \geq \pi(\oa)$, the map $G_{\pi(\ob)} \rightarrow G_{\pi(\oa)}$ is injective. The type $\tp(\oa/\pi(\oa))$ is internal, and it has a fundamental system of solutions, which can be taken to be a tuple of independent realizations of $r$. From now on, we replace $\oa$ by this tuple.

We need to prove that for any $b \models q$ independent of $\oa$, we have $b \in \dcl(\oa, \pi(b), \Pc)$. To do so, it is enough, by Fact \ref{tenzig}, to prove that any automorphism $\sigma$ of $\MM$ fixing $\oa, \pi(b)$ and $\Pc$ pointwise has to fix $b$. So consider such an automorphism $\sigma$. It restricts to $\sigma \in G_{\pi(b)\pi(\oa)}$, as it fixes $\pi(b)$ and $\pi(\oa)$. But it also fixes $\oa$, which is a fundamental system of solutions for $\tp(\oa/\pi(\oa))$. Hence, its image under the map $G_{\pi(\oa)\pi(b)} \rightarrow \pi(\oa)$ is the identity, so by collapse assumption, it is itself the identity in $G_{\pi(\oa)\pi(b)}$, and in particular fixes $b$.

\end{proof}

Note that we needed the independence assumption in order for the group $G_{\pi(\oa)\pi(b)}$ to be in the Delta groupoid. However, if the type $q$ has finite weight (see \cite{pillay1996geometric} Chapter 1, Subsection 4.4 for a definition of weight), we obtain:

\begin{prop}\label{colleq}

If the type $q$ has finite weight, the Delta groupoid associated to $q, \pi$ and $\Pc$ collapses (respectively almost collapses) if and only if there is a tuple $\oa$ of independent realizations of $r$ such that for any $b \models q$, we have $b \in \dcl(\oa, \pi(b), \Pc)$ (respectively $b \in \acl(\oa, \pi(b), \Pc)$). 

\end{prop}

\begin{proof}

Again, we will only prove the equivalence for collapsing, the other equivalence being proved in a similar way. The right to left direction is an immediate consequence of Lemma \ref{colleqind} (and does not require superstability), so we only need to prove the left to right direction. 

Assume that the Delta groupoid collapses, and let $\oa$ be a tuple of independent realizations of $r$ such that for all $b \models q$ independent of $\oa$ over $\emptyset$, we have $b \in \dcl(\oa, \pi(b), \Pc)$, it exists by Lemma \ref{colleqind}. Pick a Morley sequence $(\oa _i)_{i \in \mathbb{N}}$ in $\tp(\oa/\emptyset)$. Because the type $q$ has finite weight there is $n \in \NN$ such that for any $b \models q$, there is $i \leq n$ such that $b$ and $\oa_i$ are independent over the empty set. Let $\sigma$ be an automorphism of $\MM$ such that $\sigma(\oa_i) = \oa$, we then have that $\sigma(b)$ is a realization of $q$, independent of $\oa$. Therefore $\sigma(b) \in \dcl(\oa, \pi(\sigma(b)), \Pc)$ by Lemma \ref{colleqind}. Applying $\sigma^{-1}$, we obtain $b \in \dcl(\oa_i, \pi(b), \Pc)$. Hence, picking $\overline{\alpha} = (\oa_1 , \cdots , \oa_n)$, for any $b \models q$, we have $b \in \dcl(\overline{\alpha}, \pi(b), \Pc)$.

\end{proof}

\begin{rmk}

Recall that in a superstable theory, any type has finite weight. Hence, this proposition is true for any type in a superstable theory.

\end{rmk}

We also can prove the following proposition, which is similar to what can be obtained for internal types:

\begin{prop}\label{collfunsys}

Let $q \in S(\emptyset)$ be relatively $\Pc$-internal via the $\emptyset$-definable function $\pi$. Suppose that there is a tuple $e \in \MM$ such that for all $a \models q$, we have $a \in \acl(\pi(a),e,\Pc)$. Then the Delta groupoid $\G$ associated to $q$ and $\pi$ almost collapses.

\end{prop}

\begin{proof}

Let $a$ be a realization of $q$, independent from $e$ over the empty set. By assumption, there is a tuple $c$ of realizations of $\Pc$ such that $a \in \acl(e, \pi(a), c)$. 

Consider $\tp(ac/\acl(e))$, it is a stationary type, let $d$ be its canonical base. Pick $(a_i c_i)_{i \in \mathbb{N}}$, a Morley sequence in $\tp(ac/\acl(e))$, which we can assume to be independent from $ac$ over $e$. We know that $ac \ind_d \acl(e) $, and from this and the assumption, forking calculus yields $a \in \acl(\pi(a),c,d)$. But $d \in \acl((a_i c_i)_{1 \leq i \leq n})$ for some $n$, hence $a \in \acl(\pi(a), (a_i c_i)_{1 \leq i \leq n}, c)$, so $a \in \acl(\pi(a), (a_i)_{1 \leq i \leq n} , \Pc)$.

Now let $a' \models q$, independent from $(a_i)_{1 \leq i \leq n}$ over the empty set. Since $a$ is independent from $e$ over the empty set, and independent over $e$ of the sequence $(a_i)_{i \in \NN}$, we have that $a$ is independent from $(a_i)_{i \in \NN}$ over the empty set. Since $q = \tp(a/\emptyset)$ is stationary, this implies that $\tp(a/(a_i)_{i \in \NN}) = \tp(a'/(a_i)_{i \in \NN})$, hence $a' \in \acl(\pi(a'), (a_i)_{1 \leq i \leq n} , \Pc)$. By Lemma \ref{colleqind}, this implies that $\G$ almost collapses.

\end{proof}

As a corollary of Lemma \ref{colleqind}, we obtain the following test for internality: 

\begin{cor}
\label{collapse}
The type $q$ is internal (respectively almost internal) to $\Pc$ if and only if and only if the Delta groupoid $\G$ collapses (respectively almost collapses) and $\pi(q)$ is internal (respectively almost internal) to $\Pc$. 
\end{cor}

\begin{proof}
Once again, we will only treat the case of internality and collapse.

Suppose first that $q$ is internal to $\Pc$. We immediately get that $\pi(q)$ is internal as well.
It also yields a fundamental system of solutions, denote it $\oa$, which we can pick as a tuple of independent realizations of $q$. Moreover, we can extend $\oa$ into a tuple of independent realizations of $r$. If we now pick any $b \models q$, we have $b \in \dcl(\oa, \Pc)$, hence also $b \in \dcl(\oa, \pi(b), \Pc)$, so $\G$ collapses by Lemma \ref{colleqind}.

For the other implication, assume that $\G$ collapses and $\pi(q)$ is $\Pc$-internal. As a consequence of Lemma \ref{colleqind}, the type $q$ is internal to the family of types $\Pc \cup \lbrace \pi(q) \rbrace$. But because $\pi(q)$ is $\Pc$-internal, this implies that $q$ itself is $\Pc$-internal (see \cite{pillay1996geometric}, Remark 7.4.3).


\end{proof}

Notice that even without any internality assumption, there is always a surjective morphism $\Aut(q/\Pc) \rightarrow \Aut(\pi(q)/\Pc)$. If we assume $\pi(q)$ is $\Pc$-internal, then the target group is $\emptyset$-type-definable. 

\begin{cor}
\label{ses}

If the type $q$ is internal to $\Pc$, then there is a definable (possibly over some extra parameters) short exact sequence: \\

$1 \rightarrow \varprojlim G_{\pi(\oa)} \rightarrow \Aut(q/\Pc) \rightarrow \Aut(\pi(q)/\Pc) \rightarrow 1$\\

\noindent and the groups and morphisms are internal to $\Pc$.
\end{cor}

\begin{proof}
Set $H = \ker(\Aut(q/\Pc) \rightarrow \Aut(\pi(q)/\Pc))$. Then we have a short exact sequence: \\

$1 \rightarrow H \rightarrow \Aut(q/\Pc) \rightarrow \Aut(\pi(q)/\Pc) \rightarrow 1$ \\

Every group in this sequence is type-definable. Moreover, the left arrow is just inclusion, so is $\emptyset$-definable. As for the right arrow, if $\sigma \in \Aut(q/\Pc)$ is represented by $(\oa,\sigma(\oa))$, we can simply send it to $(\pi(\oa),\pi(\sigma(\oa)))$, so the right arrow is definable. The groups and morphisms are internal to $\Pc$. So all we need to do to finish the proof is show that $\varprojlim G_{\pi(\oa)}$ is definably isomorphic to $H$.

Since $q$ is $\Pc$-internal the Delta groupoid associated to $q, \pi$ and $\Pc$ collapses. By Corollary \ref{collapse} there is a tuple $\ob$ of realizations of $q$ such that $G_{\pi(\oc)} \rightarrow G_{\pi(\ob)}$ is injective for any $\pi(\oc) \geq \pi(\ob)$. Moreover, since $\pi(q)$ is $\Pc$-internal, we can also assume, by Remark \ref{facesurj}, that these maps are isomorphisms, hence $\varprojlim G_{\pi(\oa)} = G_{\pi(\ob)}$. By extending $\ob$ we can assume both that $\ob$ is a fundamental system of solutions for $q$ and $\pi(\ob)$ is a fundamental system of solutions for $\pi(q)$. 
We can then define a morphism $G_{\pi(\ob)} \rightarrow \Aut(q/\Pc)$ by sending $\sigma \in G_{\pi(b)}$ to $\overline{(\ob, \sigma(\ob))}$, this is well-defined because $\ob$ is a fundamental system for $q$. It is a relatively $\ob$-definable map, and it is injective, again because $\ob$ is a fundamental system for $q$. 

But $\pi(\ob)$ is a fundamental system for $\pi(q)$, so the image of this map is contained in $H = \ker(\Aut(q/\Pc) \rightarrow \Aut(\pi(q)/\Pc))$. Finally, if $\sigma \in H$, then it has to fix $\pi(\ob)$, and hence restricts to an element of $G_{\pi(\ob)}$, which yields surjectivity of $G_{\pi(\ob)} \rightarrow H$.

\end{proof}

The splitting of the short exact sequence we obtained has, in some cases, nice consequences: 

\begin{prop}
\label{splitprod}

Suppose $q$ is $\Pc$-internal and $\pi(q)$ is fundamental. If the short exact sequence: \\

$1 \rightarrow \varprojlim G_{\oa} \rightarrow \Aut(q/\Pc) \rightarrow \Aut(\pi(q)/\Pc) \rightarrow 1$ \\

\noindent is definably split and $\G_1$ is connected, then $\G_1$ is retractable. 

\end{prop}

\begin{proof}

Since $\pi(q)$ is fundamental, an element of $\Aut(\pi(q)/\Pc)$ is then defined as the class of $(\pi(a), \pi(b))$, for $\pi(a),\pi(b)$ two realizations of $\pi(q)$. Let $s$ be a section of the short exact sequence. We can then define $g_{\pi(a), \pi(b)} = s((\pi(a),\pi(b))/E')$, where $E'$ is the equivalence relation used to define $\Aut(\pi(q)/\Pc)$. This is uniformly $\emptyset$-definable, and the compatibility condition is easily checked. 

\end{proof}

We hence obtain a partial converse to Theorem \ref{retgrp}: 

\begin{theorem}

Suppose $q$ is $\Pc$-internal. Assume $\G_1$ is connected, and $\pi(q)$ is fundamental. Then $\G_1$ is retractable if and only if the short exact sequence: \\ 

$1 \rightarrow \varprojlim G_{\oa} \rightarrow \Aut(q/\Pc) \rightarrow \Aut(\pi(q)/\Pc) \rightarrow 1$ \\

\noindent is definably split.

\end{theorem}

We have seen that internality of $q$ can be read from the collapse of the Delta groupoid. It is also linked to the following notion, first introduced in \cite{moosa2010model}: 

\begin{definition}
Suppose $q \in S(d)$ is a stationary type, and $\Pc$ is a family of partial types, over the empty set. We say that $q(x)$ preserves internality to $\Pc$ if whenever $a \models q$ and $c$ are such that $\tp(d/c)$ is almost $\Pc$-internal, then $\tp(a/c)$ is also almost $\Pc$-internal.
\end{definition}

We want to obtain a sufficient condition for $\tp(a/d)$ to preserve internality. Note that by setting $c=d$, we get that preserving internality implies almost internality. 

Remark that if $\tp(a/d) $ is $\Pc$-internal and stationary, then we can consider the type $p = \tp(ad/\emptyset)$, and the projection $\pi$ on the $d$-coordinate. This is a projection with $\Pc$-internal fibers, so yields an $\emptyset$-type-definable Delta groupoid $\G_{p}$.

Intuitively, collapse of the groupoid associated to $\pi$ and $q$ means that the only thing missing for $q$ to be $\Pc$-internal is for $\pi(q)$ to be $\Pc$-internal. Therefore, the following result appears quite natural: 

\begin{prop}\label{collimppres}
Suppose $q = \tp(a/d)$ is $\Pc$-internal and stationary. Let $p = \tp(ad/\emptyset)$. If the Delta groupoid $\G$ associated to $p$ and the projection $\pi$ on the  $d$-coordinate almost collapses, then $\tp(a/d)$ preserves internality to $\Pc$.
\end{prop}

\begin{proof}

Recall that we assume $\emptyset = \acl(\emptyset)$, hence $p$ is stationary. Lemma \ref{colleqind} implies the existence of a tuple $\oee$ of realizations of $p$, independent from $ad$ over $\emptyset$, such that $ad \in \acl(\oee, d, \Pc)$. Taking a realization of $\tp(\oee /ad)$ independent from $c$ over $ad$, we can assume that $\oee$ is independent from $adc$ over $\emptyset$.

Now, the type $\tp(d/c)$ is almost $\Pc$-internal, hence there is a tuple $\od$ of realizations of $\tp(d/c)$, independent from $d$ over $c$, such that $d \in \acl(\od , c , \Pc)$. We can assume, without loss of generality, that $\od$ is independent from $ad \oee$ over $c$. Forking calculus yields that $\od e$ is independent from $ad$ over $c$. But $ad \in \acl(\oee ,d , \Pc)$ and $d \in \acl(\od , c ,\Pc)$, so $ad \in \acl (\oee, \od, c , \Pc)$. Hence $\tp(ad/c)$ is almost $\Pc$-internal.

\end{proof}

The converse to this proposition is likely to be false. Indeed, suppose $\tp(a/d)$ is $\Pc$-internal and stationary, but for any tuple $c$, the type $\tp(d/c)$ is almost $\Pc$-internal if and only if it is algebraic. This implies that $\tp(a/d)$ preserves internality to $\Pc$, but should not imply that the groupoid associated to $\tp(ad/\emptyset)$ collapses. The construction given on top of page 4 of \cite{moosa2010model} is a good candidate for a counterexample. It would be interesting to find a necessary and sufficient condition, in terms of Delta groupoids, for a type to preserve internality.

In the literature, examples of types preserving internality appear in \cite{chatzidakis2012note}, \cite{chatzidakis2015differential} and \cite{moosa2010model}. A potential direction for future work is to examine, for each of these examples, if the collapse of a Delta groupoid is involved or not.

\end{document}